\documentclass[11pt]{article}
\title{
 The exact asymptotic of the  collision  time
tail distribution  for independent Brownian particles with different drifts.}
\author{Zbigniew Pucha\l a$^{1,2,3}$ and  Tomasz Rolski$^{1,2}$ }

\usepackage{a4wide}
\usepackage{amssymb}
\usepackage{amsmath}
\usepackage{amsfonts}

\newcommand{\makepcx}[3]
{\setlength{\unitlength}{1cm}
\begin{picture}(#1,#2)(0,0)
\put(0,#2){\special{em:graph #3.pcx}}
\end{picture}
}

\newcommand{\centerpcx}[3]
{\begin{center}
  \setlength{\unitlength}{1cm}
  \begin{picture}(#1,#2)(0,0)
     \put(0,#2){\special{em:graph #3.pcx}}
  \end{picture}
 \end{center}
}

\newcommand{\Prob}{{\rm I\hspace{-0.8mm}P}}
\newcommand{\Exp}{{\rm I\hspace{-0.8mm}E}}

\newcommand{\iz}{{\rm \rlap Y\kern 2.2pt Y}}

\newcommand{\bbR}{{\rm I\hspace{-0.8mm}R}}

\newcommand{\bbZ}{\mathbb{Z}}

\newcommand{\bfA}{\boldsymbol A}

\newcommand{\bfU}{\boldsymbol U}

\newcommand{\bfX}{\boldsymbol X}

\newcommand{\bfa}{\boldsymbol a}

\newcommand{\bfm}{\boldsymbol m}

\newcommand{\bff}{\boldsymbol f}

\newcommand{\bfs}{\boldsymbol s}

\newcommand{\bfk}{\boldsymbol k}

\newcommand{\bfx}{\boldsymbol x}
\newcommand{\bfy}{\boldsymbol y}
\newcommand{\bfz}{\boldsymbol z}

\newcommand{\bfalpha}{\boldsymbol \alpha}

\newcommand{\bfxi}{\boldsymbol \xi}

\newcommand{\calF}{{\cal F}}

\newcommand{\refs}[1]{(\ref{#1})}

\newcommand{\proof}{\noindent {\it Proof.\ }}
\newcommand{\halmos}{\newline\vspace{3mm}\hfill $\Box$}
\newcommand{\ind}{1\hspace{-1mm}{\rm I}}

\renewcommand{\theequation}%
{{\rm
\arabic{section}.\arabic{equation}}}

\newcommand{\diag}{\mbox{\rm diag}}

\newcommand{\ud}{{\,{\rm d}}}
\newcommand{\Van}{\Delta}

\newcounter{mylistcnt}
\renewcommand{\themylistcnt}{{\rm({\roman{mylistcnt}})}}

\newcounter{zad}

\newtheorem{Th}{Theorem}[section]

\newtheorem{Prop}[Th]{Proposition}
\newtheorem{Lemma}[Th]{Lemma}
\newtheorem{Ex}{Example}[section]

\newtheorem{Defin}[Th]{Definition}
\newtheorem{D}{}[section]
\newtheorem{Rem}{Remark}[section]

\begin{document}
\maketitle
\stepcounter{footnote}\footnotetext{ Mathematical Institute,
University of Wroc\l aw, pl. Grunwaldzki 2/4, 50-384 Wroc\l aw,
Poland} 
\stepcounter{footnote}\footnotetext{This work  was partially
supported by
a Marie Curie Transfer of Knowledge Fellowship of the
European Community's Sixth Framework Programme: Programme HANAP under
contract number MTKD-CT-2004-13389. }
\stepcounter{footnote}\footnotetext{This work  was partially
supported by This work  was partially
supported by KBN Grant N201 049 31/3997 (2007).}
\begin{abstract}
In this note we consider the time of the collision $\tau$ for
$n$ independent  Brownian motions  $X^1_t,\ldots,X_t^n$
with drifts $a_1,\ldots,a_n$, each starting from $\bfx=(x_1,\ldots,x_n)$,
where $x_1<\ldots<x_n$. We show the exact  asymptotics
of $\Prob_{\bfx}(\tau>t) = Ch(\bfx)t^{-\alpha}e^{-\gamma t}(1 + o(1))$ as $t\to\infty$
and identify $C,h(\bfx),\alpha,\gamma$ in terms of the drifts.
\end{abstract}

\newpage
\section{Introduction and results}
Let $W=\{\bfy: y_1<\ldots<y_n \}$ be the Weyl chamber.
 Consider $\bfX_t=(X^1_t, \ldots, X^n_t)$, wherein coordinates
are  independent Brownian motions with unit variance parameter, drift vector
$\bfa=(a_1,\ldots,a_n)$ and starting point $\bfX_0=\bfx\in W$.
In this paper we study the collision time $\tau$, which is the
exit time of $\bfX_t$ from the Weyl chamber, i.e.
$$\tau=\inf\{t>0:\bfX_t\notin W\}\;.$$
For identical drifts $a_1=\ldots =a_n$ , say $a_i\equiv 0$,
the celebrated Karlin-McGregor formula states (see \cite{karlinmcgregor})
\begin{eqnarray}\label{KarlinMcGregor}
\Prob( \tau > t ; \bfX_t \in \ud \bfy) = \det \left[p_t(x_i,y_j)\right]\ud \bfy \; ,
\end{eqnarray}
where $p_t(x,y) =
\frac{1}{\sqrt{2 \pi t}} e^{- \frac{(x-y)^2}{2t}}$, which yields
the tail distribution
of $\tau$:
$$\Prob_{\bfx}(\tau>  t)=\int_W \det \left[p_t(x_i,y_j)
\right] \,\ud \bfy\;.$$
For the use of  Karlin-McGregor formula it is essential that processes
$X_t^1,\ldots,X^n_t$ are independent copies of the same strong Markov, with skip-free realizations
process, starting at $t=0$ from $\bfx\in W$.
In this case the asymptotic of  $\Prob_{\bfx}(\tau>  t)$ was first studied by
Grabiner \cite{grabiner} (for the Brownian case) (see also  proofs by
Doumerc and O'Connell \cite{doumerc} and Pucha\l a~\cite{puchala}) Later
 Pucha\l a \& Rolski \cite{puchalarolski}) showed that
this asymptotic is also true for the Poisson and continuous time random walk case.
The above mentioned asymptotics is:
\begin{equation}\label{glowna.asymptotyka}
\Prob_{\bfx}(\tau>t)\sim D\Van(\bfx)t^{-n(n-1)/4},
\end{equation}
where
$\Van(\bfx)=\det \left[\left(x_i^{(j-1)}\right)_{i,j=1}^n\right]$ is the Vandermonde determinant, and
\begin{equation}\label{stalaC}
D=\left( 2\pi \right) ^{-n/2} c_n \int_{W}e^{-\frac{%
\left\vert \bfy\right\vert ^{2}}{2}} \Van \left( \bfy\right) d\bfy\;,
\end{equation}
  for $t\to\infty$. Here and below
$1/c_n=\prod_{j=1}^{n-1}j!$.

In this note we study the same problem, however for Brownian motions with different drifts.
For this we derive first, in Section \ref{s.formula}, a formula for $\Prob_{\bfx}(\tau>t)$ by the change of  measure.
It is apparent that possible results must depend on the form of drift vector $\bfa$. For example we can
analyze all cases for $n=2$, because in this case the collision equals to the first passage to zero
of the Brownian process $X^2_t-X^1_t$, for which the density function
is known (see e.g. \cite{borodinsalminen}). Hence
$$
\Prob_{\bfx}(\tau>t)= \int_{t/2}^{\infty} \frac{x}{\sqrt{2 \pi} s ^{3/2}} \exp \left[ - \frac{(x + a s)^2}{2 s}\right] \ud s\;,
$$
where $x=x_2-x_1$ and $a=a_2-a_1$. This yields
\begin{eqnarray*}
\Prob_{\bfx}(\tau>t)=
\left\{
\begin{array}{ll}
 \frac{ 2^{5/2}}{a^2\sqrt{2 \pi}} x e^{a x} t^{-3/2} e^{-t a^2 /4 } \, \,(1 + o(1))\;,  & a_1>a_2\\[0.3cm]
\frac{2^{\frac{3}{2}} }{\sqrt{2 \pi}} x t^{-\frac{1}{2}}\, \, (1 + o(1))\;,   & a_1=a_2\\[0.3cm]
1 - e^{-a x} + o(1)  & a_1<a_2\;.
\end{array}
\right.
\end{eqnarray*}

For general $n$ the situation is much more complex and different scenarios are possibles. For example the drifts can be diverging
and then $\Prob_{\bfx}(\tau>t)$ tends to a  positive constant, which the situation was analyzed by
Biane {\it et al} \cite{bianeetal}. Another case is when all drifts are equal, in which the case the probability $\Prob_{\bfx}(\tau>t)$
is polynomially decaying, as it was found by Grabiner \cite{grabiner}. However there are various situations when the probabilities
are exponentially decaying with polynomial prefactors. The full characterization depends on a concept of
the stable partition
of the drift vector, which the notion is introduced in Section \ref{s.stable}. In Section \ref{s.main} we state the main theorem,
which shows all possible exact asymptotics of  $\Prob_{\bfx}(\tau>t)$ in from of $Ch(\bfx)t^{-\alpha}e^{-\gamma t}$,
where formulas for $C$,$\alpha$ and $\gamma$ are given in terms of the stable partition of the drift vector.

\section{Formula for $\Prob_{\bfx}(\tau>t)$.}\label{s.formula}
We note our basic probabilistic space with natural history filtration $(\Omega , \calF, (\calF_t) , \Prob_{\bfx})$ and
consider on it process $\bfX_t$ as defined in the Introduction.
Unless otherwise stated we tacitly assume that $\bfx \in W$.
We start off a lemma on the change of measure for the Brownian case,
which the proof can be found for example in Asmussen \cite{APQ},
Theorem 3.4.
Let $M_{t} = e^{<\bfalpha,\bfX_t>}/\Exp e^{<\bfalpha,\bfX_t>}$ be a Wald  martingale.
For a probability measure $\Prob_{\bfx}$ its restriction to $\calF_t$ we denote by $\Prob_{\bfx|t}$.
Let $\tilde{\Prob}_{\bfx}$ be a probability measure
obtained by the change of measure $\Prob_{\bfx}$ with the use of martingale
$M_t$, that is defined by a family of measures $\tilde{\Prob}_{\bfx|t}=M_t\,\ud \Prob_{\bfx|t}$,\ $t\ge0$.
For the theory we refer e.g. to  Section XIII.3 in \cite{APQ}
\begin{Lemma}\label{l.zamiana} If $\bfX_t$ is a Brownian motion with drift $\bfa$ under $\Prob_{\bfx}$, then
this process is a Brownian motion with drift $\bfa+\bfalpha$ under $\tilde{\Prob}_{\bfx}$.
\end{Lemma}

The sought for formula for the tail distribution of the collision time is given
in the next proposition.
\begin{Prop}\label{GBM.Prop.assym.po.zamianie}
\begin{eqnarray}\label{GBM.Prop.assym.po.zamianie.wzor}
\lefteqn{\Prob_{\bfx}(\tau>t) = } && \nonumber \\
&=&
(2\pi)^{-n/2}
e^{-<\bfa,\bfx> -||\bfx||^2/2t}
\int_W e^{-||\bfy - \bfa \sqrt{t}||^2 /2}
\det[e^{x_iy_j/\sqrt{t}}]\ud\bfy\;.
\end{eqnarray}
\end{Prop}
\proof
We use  $\bfalpha = - \bfa$ to eliminate the drift under $\tilde\Prob_{\bfx}$.
Thus $\Prob_{\bfx}(\tau>t) =
\tilde{\Exp}_{\bfx}[M^{-1}_t;\tau>t]\;$.
Now by  Karlin-McGregor formula \refs{KarlinMcGregor} we  write
\begin{eqnarray*}
\Prob_{\bfx}(\tau>t) &=&
\tilde{\Exp}_{\bfx}[e^{<\bfa,\bfX_t>}\Exp_{\bfx}
e^{<-\bfa,\bfX_t>};\tau>t]\\
&=&
e^{<-\bfa,\bfx>-||\bfa||^2 t/2}
\int_{\bfy\in W} e^{<\bfa,\bfy>}
\det[p_t(x_i,y_j)]\ud\bfy \;,
\end{eqnarray*}
and next, algebraic manipulations yield \refs{GBM.Prop.assym.po.zamianie.wzor}.
\halmos

In the paper we use the following
vector notations. For a vector $\bfa\in\bbR^n$ we denote
$\bfa_{[i,j]}=(a_i,a_{i+1},\ldots,a_j)$ and
$\bar{a}_{[i,j]}=(a_i+a_{i+1}+\ldots+a_j)/(j-i+1)$.
We also use $\bfa_{(i,j]}=(a_{i+1},\ldots,a_j)$ and $\bfa_{(i,j)}=(a_{i+1},\ldots,a_{j-1})$.
By $\bfz^{\bfk}$, where $\bfz=(z_1,\ldots,z_m)$ and $\bfk=(k_1,\ldots,k_m)$ we denote $\prod_{j=1}^m z_j^{k_j}$.

\section{Stable partition of $\bfa$.}\label{s.stable}
Let $\bfa\in\bbR^n$. Our aim is to make a suitable partition
\begin{equation}\label{eq.partition}(a_1,\ldots,a_{\nu_1})(a_{\nu_1+1},\ldots,a_{\nu_1+\nu_2}),\ldots,
(a_{\nu_1+\ldots+\nu_{q-1}+1},\ldots,a_{\nu_1+\ldots+\nu_q})\;.
\end{equation}
of $\bfa$, where $\nu_i>0$.
For short we denote $m_1=\nu_1,m_2=\nu_1+\nu_2,\ldots,m_q=\nu_1+\ldots+\nu_q=n$. We also set $m_0 = 0$.

 We say that sequence $\bfa$ is {\em irreducible}   
if
\begin{equation} \label{nierownosci.speed.mass.lemat}
\left.
\begin{array}{ccc}
\bar{a}_{[1;1]} &>& \bar{a}_{[2;n]} \\
\bar{a}_{[1;2]} &>&\bar{a}_{[3;n]}  \\
\vdots&\vdots& \vdots\\
\bar{a}_{[1;n-1]}  &>& \bar{a}_{[n;n]}
\end{array}
\right\}\;.
\end{equation}

Suppose we have a partition defined by $m_1,\ldots,m_q$\,. The mean of
the $i^{\text{th}}$ sub-vector is denoted by $f^i=\bar{a}_{(m_{i-1};m_{i}]}$.
Furthermore we define a vector
$\bff = (f_1 , \dots , f_n)$ by
$$
f_i =f^k, \quad{\rm if}\quad  m_{k-1} < i \leq m_k\;.
$$
It is said that  partition \refs{eq.partition} of vector $\bfa$ is {\em stable} if
\begin{equation}\label{warunek1}
f^1\le f^2\le \ldots \le f^{q}
\end{equation}
 and each vector
$\bfa_{(m_{i-1},m_i]}$ is irreducible $(i=1,\ldots,q)$.
Remark that a stable partition is defined if we know $\bfm= (m_1,\ldots,m_q)$ for which \refs{warunek1} hold
and each $\bfa_{(m_{i-1},m_i]}$ is irreducible $(i=1,\ldots,q)$.
In the sequel, for a given stable partition
of $\bfa$, characters $q,\bff, \bfm$ are reserved for it.

Consider now $f_{m_1},f_{m_2},\ldots,f_{m_q}$
and define a subsequence $\bfm^{'}=(m'_1,\ldots,m{'}_{q'})$ of $\bfm=(m_1,m_2, \dots,m_q)$
as follows. Let $q'$ be the number of strict inequalities in $f^1\le f^2\le\ldots\le f^q$
plus 1. Furthermore
we define inductively by
$m'_0=0$
and for $i=1,\ldots,q'-1$
$$
m'_i=\inf\{m_{j}>m'_{i-1}:\ m_j\in\bfm, f_{m_j}< f_{m_{j+1}}\}\;.
$$
and finally we set $m_{q'}=n$.
We also define a subsequence of indices $i_0 , i_1 , \dots , i_{q'}$ inductively by
$i_0 = 0$
and
$$
i_k = \inf \{j > i_{k-1}:\ f_{m_j} < f_{m_{j+1}}\}.
$$
Hence we have
$$
f_{m'_{1}} < f_{m'_{2}} < \dots < f_{m'_{q'}}.
$$
In this case we say that $ (m'_{1}, \dots, m'_{q'})$ is a strong representation of the stable partition of $\bfa$
and $q', \ (m'_{1}, \dots, m'_{q'})$ are characters reserved for it.
Set $\nu'_i=m'_{i}-m'_{i-1}$, $(i=1,\ldots,q')$.

\medskip
\noindent {\bf Example 1}
Suppose that $\bfa=(3,1,2,5,1)$. Then $q=3$ and $m_1=2,m_2=3,m_3=5$ define the stable partition
$(3,1)(2)(5,1)$ with means $f^1=2,f^2=2,f^3=3$. furthermore $q'=2$, $m_1'=3,m_2'=5$
and $i_1=2,i_2=5$.

\begin{Prop}\label{istnienie.grupy} For each vector $\bfa$, there exists its unique stable partition.
\end{Prop}

\medskip
Before we state a proof of Proposition \ref{istnienie.grupy} we prove few lemmas.

\begin{Lemma}\label{L1}
If $\bfa=(a_1,\ldots,a_n)$ is irreducible, then
\begin{equation} \label{nierownosci.z.fi}
\left.
\begin{array}{ccccc}
\bar{a}_{[1;1]} &>&f_n&>& \bar{a}_{[2;n]} \\
\bar{a}_{[1;2]} &>&f_n&>& \bar{a}_{[3;n]}  \\
\vdots&\vdots& \vdots&\vdots&\vdots\\
\bar{a}_{[1;n-1]}  &>&f_n&>&  \bar{a}_{[n;n]}
\end{array}
\right\}\;.
\end{equation}
\end{Lemma}

\proof $f_n$ is a nontrivial weighted mean of every pair $\bar{a}_{[1;i]}$ and $\bar{a}_{[i+1;n]}$.
\halmos

\begin{Lemma}\label{lemma.przejscieSredniej}
In a stable partition, for each element  $\bfa_{(m_{i-1};m_i]}$
$$\bar{a}_{(m_{i-1};m_{i-1}+k]}\ge f_{m_i}\;.$$
\end{Lemma}

\proof
The case $k\le m_i-m_{i-1}$ follows from Lemma \ref{L1}.
Clearly for $k= m_i-m_{i-1}$ we have equality.
Consider now $k > m_{i}-m_{i-1}$. Than  $\bar{a}_{(m_{i-1};m_{i-1}+k]}$ is
  a weighted mean of
$f_{m_i}$ and $\bar{a}_{(m_i;m_{i}+k-(m_i - m_{i-1})]}$
 and the later term is  greater or equal than $f_{m_i}$ by \refs{warunek1} and \refs{nierownosci.z.fi}.
\halmos

In the next lemma we consider two vectors $\bfa_1\in\bbR^{n_1}$ and $\bfa_2\in \bbR^{n_2}$.
The corresponding $f$-s are $f_{n_1}$ and  $f_{n_2}$ respectively.
We consider a situation of creating a new vector $(\bfa_1,\bfa_2)=(a_{1}\ldots,a_{n_1+n_2})\in\bbR^{n_1+n_2}$.

\begin{Lemma}\label{lemma.laczenieGrup}
Suppose that $\bfa_1$ and $\bfa_2$ are irreducible and  $f_{n_1} > f_{n_2}$.
Then vector $(\bfa_1,\bfa_2)$ is irreducible.
\end{Lemma}

\proof
Recall that $(\bfa_1,\bfa_2)=(a_{1}\ldots,a_{n_1+n_2})\in\bbR^{n_1+n_2}$.
Suppose $1\le k\le n_1$. By Lemma \ref{L1} we have $\bar{a}_{[1;k-1]}>f_{n_1}> \bar{a}_{[k;n_1]}$, also
 $\bar{a}_{[1;k-1]}>f_{n_1}> f_{n_2}$.
Hence $\bar{a}_{[1;k-1]}>
\bar{a}_{[k;n_1+n_2]}$ becasue $ \bar{a}_{[k;n_1+n_2]}$ is a weighted mean of $\bar{a}_{[k;n_1]}$ and $f_{n_2}$.
Suppose now $n_1<k$. Then $\bar{a}_{[1;k-1]}$ is a weighted mean of $f_{n_1}$ and
$\bar{a}_{[n_1+1,k]}$ and both by Lemma \ref{L1} are greater than $\bar{a}_{[k;n_1+n_2]}$, which completes the proof.
\halmos

\medskip\noindent{\it Proof of}\ Proposition \ref{istnienie.grupy}.
The existence part is by induction with respect $n$.
For $n = 2$ we have two situations \\
\begin{enumerate}
\item if $a_1 \leq a_2$, than  $q=2$ with $m_1 = 1, \,\, m_2 = 2$ is a stable partition,
\item if $a_1 > a_2$, than $q=1$ with $m_1=2$ is a stable partition.
\end{enumerate}
Assume that there exists a \emph{stable} partition with  $q$ partition vectors of a vector $\bfa \in R^n$.
We add a new element $a_{n+1}$ at the end of vector $\bfa$ to create new one $(\bfa,a_{n+1})=(a_1,\ldots,a_{n+1})$.

We have two  situations.
\begin{enumerate}
\item If $a_{n+1} \geq f^{q}$ than in a stable partition  $a_{n+1}$ is alone in the $q+1$ partition vector.
\item  If $a_{n+1} < f^{q}$, than we proceed inductively as follow.
We use Lemma \ref{lemma.laczenieGrup} with $\bfa_1=\bfa_{[m_{q-1};m_q]}$ and $\bfa_2=(a_{n+1})$ and let $f^q$ and
$f^{q+1}=a_{n+1}$ are means of these partition vectors.
In result $(\bfa_{(m_{q-1};m_q]},a_{n+1})$ form an irreducible vector, for which we have to check
 whether condition \refs{warunek1} holds. If yes, then   we end with a stable partition, otherwise
we join the  $q-1$ partition vector with the new $q$ partition vectors and repeat the procedure. In the worst case we end up with one partition vector.
\end{enumerate}

For the uniqueness proof ,
suppose that we have two different \emph{stable} partitions: $m_1^1 < m_2^1 < \dots <m_{q_1}^1$
and $m_1^2 < m_2^2 < \dots <m_{q_2}^2$. The means of $f$s are
$(f^1)^1,\dots,(f^1)^{q_1}$

for the first partition vector and
$(f^2)^1,\dots,(f^2)^{q_2}$ for the second respectively.
Since partitions are supposely different,
 there exists $i$ such that $m_i^1 \neq m_i^2$.  We take the minimal $i$ with this property and without loss of
generality we can assume $m_i^2 > m_i^1$. Set $k = m_i^2 - m_i^1$.
We have to analaze the following cases.
\begin{enumerate}
\item $(m_i^2=m_{i+1}^1)$. We have
$$\bar{a}^2_{[m^2_{i-1}+1;k]}>\bar{a}^2_{[k+1;m^2_{i}]}$$
On the other hand  $(f^1)^i=\bar{a}_{[m^1_{i-1}+1;m_i^1]}=\bar{a}_{[m^2_{i-1}+1;m_i^1]}>
\bar{a}_{[m^1_{i}+1;m_i^2]}=\bar{a}_{[m^1_{i}+1;m_{i+1}^1]}=
(f^2)^i$ 
 and this contradics with $(f^1)^i\le (f^1)^{i+1}$.
\item $(m_i^2>m_{i+1}^1)$. We have
$\bar{a}_{[m^2_{i-1}+1;m_i^1  ;m_i^1+1;m^1_{i+1}   ]}$ and
by Lemma \ref{lemma.przejscieSredniej}
$$(f^1)^i=\bar{a}_{[m^2_{i-1}+1;m_i^1]}>\bar{a}_{[m^1_{i}+1;m_i^2]}\ge (f^1)^{i+1}\;,$$
 which is a contradiction.
\item $(m_i^2<m_{i+1}^1)$. We have by Lemma \ref{lemma.przejscieSredniej}
 $$(f^1)^i=\bar{a}_{[m^1_{i-1}+1;m_i^1]}>\bar{a}_{[m^1_{i}+1;m_i^2]}\ge (f^1)^{i+1}\;,$$
 which is a contradiction.
\end{enumerate}
The proof is completed.
\halmos

\medskip\noindent{\bf Remark}
The stable partition can be obtained by considering the following
simple deterministic dynamical system. We have $n$ particles starting from
$x_1 < x_2 < \dots < x_n$. The $i^{\text{th}}$ particle has  speed $a_i$.
Each particle moves with a constant speed on the real line until
it collides with one of its neighboring   particle (if it happens).
Then both the particles coalesce and from this time on they move with the proportional speed which is the mean of
speed of colliding particles, and so on.
Ultimately the particles will form never colliding groups, which are the same as
in the stable partition of $\bfa$. Notice that resulted grouping do not
depend on a starting position  $\bfx$.

\section{The theorem and examples.}\label{s.main}
We begin introducing some notations.
Suppose that $\bfa$ has  a stable partition with characteristics $q,(m_i),q',(m'_{i})$ respectively.
In the sequel we will use the following notations:
\begin{eqnarray}\label{GBM.gamma}
\gamma = \frac{1}{2} \sum_{\ell=1}^q \left( \frac{1}{\nu_\ell} \sum_{m_{l-1}<u<v\leq m_l} (a_u - a_v)^2 \right),
\end{eqnarray}

\begin{eqnarray}\label{GBM.alpha}
\alpha = \frac{1}{2} \left(
\sum_{j=1}^{q'} \binom{\nu'_{j}}{2}
+ (n-q) + \sum_{j=1}^q \binom{\nu_j}{2}
\right),
\end{eqnarray}
\begin{eqnarray}\label{GBM.h()}
h(\bfx)&=&
e^{-<\bfx,\bfa>} \det\left[ e^{x_i f_j} x_i^{\sum_{l=1}^{q'} (j-m'_{l-1} - 1)  \ind_{\{m'_{l-1} < j \leq m'_{l}\}}} \right] \;.
\end{eqnarray}
Moreover we define a function
\begin{eqnarray}\label{GBM.I()}
\lefteqn{I(\bfa,t)}\nonumber\\
&=&\int_{W-\bff\sqrt{t}}
e^{-\frac{1}{2}|\bfz|^2}
e^{-\frac{1}{2} \sum_{l=1}^q \left( \frac{2\sqrt{t}}{\nu_l} \sum_{m_{l-1}<u<v\leq m_l} (z_u-z_v)(a_v - a_u) \right)}
\prod_{j=1}^{q'}\Delta(\bfz_{(m'_{j-1};m'_j]})
 \ud\, \bfz\;.\nonumber\\
\end{eqnarray}

Remark that from Lemma \ref{lemma.y^2_cz1} it will follow
$$ \frac{1}{\nu_l}\sum_{m_{l-1}<u<v\leq m_l} (a_u - a_v)^2= \sum_{m_{l-1}<u<v\leq m_l} (a_u - \bar{a}^l)^2,$$
where
$$\bar{a}^l= \frac{1}{\nu_l}\sum_{u=m_{l-1}+1}^{m_l} a_u\;.$$

Using this notation we now state a proposition which is useful for calculations in some cases.
\begin{Prop}\label{GBM.MainProp}
\begin{eqnarray} \label{GBM.MainPropWzor}
\Prob_{\bfx}(\tau>t) &=&
 (2\pi)^{-n/2}\prod_{j=1}^{q{'}}c_{\nu'_j}
e^{-\gamma t}  t^{-\frac{1}{2} \sum_{j=1}^{q'} \binom{\nu'_j}{2}}  \nonumber \\
&&\times
 e^{-<\bfx,\bfa>}\det \left[e^{x_k f_j} x_k^{\sum_l^{q'} (j-m_{i_{l-1}} - 1)  \ind_{\{m_{i_{l-1}} < j \leq m_{i_l} \}}} \right]
\nonumber\\
&&\times
I(\bfa,t)\, \,  (1 + o(1)).   
\end{eqnarray}
\end{Prop}

Remark that formula \refs{GBM.MainPropWzor}
 does not give us straightforward asymptotic because integral $I(\bfa,t)$ depends on~$t$.
However in some cases this dependence vanishes
and this is why Proposition \ref{GBM.MainProp} can be sometimes useful.

 The next theorem gives us asymptotic for all cases.
\begin{Th}\label{GBM.MainTh}
For some $C$ given below, as $t \rightarrow \infty $
\begin{eqnarray*}
\Prob_{\bfx}(\tau>t)
&=&
C h(\bfx) t^{- \alpha} e^{-\gamma t} (1 + o(1)),
\end{eqnarray*}
$\gamma$, $\alpha$, and $h(\bfx)$ are defined in \refs{GBM.gamma},\refs{GBM.alpha},\refs{GBM.h()} respectively.
\end{Th}

To show $C$ we need few more definitions. Let
\begin{equation}\label{def.H}
H(s_1,\ldots,s_{\ell})=\prod_{1\le i\le j\le \ell+1}(s_i+\ldots+s_{j-1}).
\end{equation}
Define now
\begin{eqnarray}\label{GBM.C}
C =A_1\times A_2\times A_3\;,
\end{eqnarray}
where
$$A_1=(2\pi)^{-n/2} \sqrt{2\pi n}
\prod_{j=1}^{q'}c_{\nu'_j},
$$
\begin{eqnarray*}
\lefteqn{A_2=\idotsint\limits_{\xi_i>0 : i \notin \{m_1,\dots , m_q\} }
e^{-\frac{1}{2}\sum_{l=1}^q \left( \frac{2}{\nu_{l}} \sum_{m_{l-1}<u<v\leq m_l} (\xi_u + \dots + \xi_{v-1})(a_u - a_v) \right)}}
\nonumber \\
&&\hspace{4.0cm} \times \prod_{i=1}^q H \left(\bfxi_{(m_{i-1} ;m_i - 1)}\right)
\prod_{ i \notin \{m_1,\dots , m_q\}}\ud \xi_i \;,
\end{eqnarray*}
and
\begin{eqnarray*}
\lefteqn{A_3=\idotsint\limits_{\xi_i > 0: i \in \{m_1, \dots , m_q \} \setminus \{ m_{l_1}, \dots , m_{l_{q'}} \} }
\quad\idotsint\limits_{\xi_i > - \infty : i \in \{ m_{l_1}, \dots , m_{l_{q'}} \} }
e^{-\frac{1}{2}
\left(
    \sum_{k,l\in \{m_1, \dots, m_q\}} S_{kl} \xi_k \xi_l
\right)}}
\nonumber \\
&&\hspace{2.6cm}\times
\prod_{k=0}^{q-1}
\prod_{i : \{i,i+1,\dots , i+k\} \atop \in \{1, \dots , q \} \setminus \{ {l_1}, \dots , {l_{q'}} \}}
\left(
    \sum_{j=0}^{k} \xi_{m_{i+j}}
\right)^{\nu_i \nu_{i+k+1}}
\prod_{   i \in \{m_1,\dots , m_q\}  }\ud \xi_i\;,
\end{eqnarray*}
where $S_{kl}=(n-2)k$ for $k\le l$ and $S_{kl}=S_{lk}$.
In the remaining part of this section
we diplay some special cases.

\bigskip

\noindent {\bf Example 2}
$(a_1 = a_2 = \dots = a_n)$ This is no drift case.
Here $q = n$ and $m_1 = 1, m_2 = 2 , \dots , m_n = n$, also $q' = 1$ and $m'_{1} = n$.
In result $f_{m_1} = a_1 , \dots , f_{m_n} = a_n$.
Let $a$ be the common value of the drift.
Using Proposition \ref{GBM.MainProp} we have
\begin{eqnarray*}
\Prob_{\bfx}(\tau>t) &=&
(2\pi)^{-n/2}c_n
e^{-<\bfx,\bfa>}
\det \left[e^{x_k f_j} x_k^{j-1} \right]
t^{-\frac{1}{2} \binom{n}{2}}\\
&&\times
\int_{W-\bff\sqrt{t}}
e^{-\frac{|\bfz|^2}{2}}
\Van_{n}(\bfz_{[1;n]})  \ud \bfz \,(1 + o(1)).
\end{eqnarray*}
First we notice that since all the coordinates in vector $\bff$ are the same, we have
$$
\det \left[e^{x_k f_j} x_k^{j-1} \right]  = e^{<\bfx,\bff>} \det \left[x_k^{j-1} \right] = e^{<\bfx,\bfa>} \Delta (\bfx).
$$
Furthermore $W - \bff \sqrt{t} = W$ because
$y_1 < y_2 < \dots < y_n$  if and only if   $ y_1 + a\sqrt{t} < y_2 + a\sqrt{t} < \dots < y_n + a\sqrt{t}$.
Finally we write
\begin{eqnarray*}
\Prob_{\bfx}(\tau>t) &=& C \; h(\bfx) \; t^{- \alpha} \; (1 + o(1)),
\end{eqnarray*}
where
\begin{eqnarray*}
&\alpha&=\ \frac{1}{2} \binom{n}{2},
\\
&h(\bfx)&=\ \Van_{n}(\bfx),
\\
&C&=\ (2\pi)^{-n/2}c_n
\int_{W} e^{-\frac{|\bfz|^2}{2}}\Delta(\bfz) \ud \bfz .
\end{eqnarray*}

\bigskip
Before we state the next example we prove the following lemma.
\begin{Lemma}\label{lemmaW->Rn}
If $\bfa \in W$, 
then $\{W - \bfa t \} \rightarrow \bbR^n$ as $t \rightarrow \infty$.
\end{Lemma}
\begin{proof}
Let $\bfa \in W$. We show that for all   ${\bfy \in \bbR^n}$ there exists ${s > 0}$, such that for
all ${t > s}$,  $\bfy \in \{W - \bfa t \}$.
Let $y \in \bbR^n$. We note  $b_i = y_{i+1} - y_i$ and $d_{i}
= a_{i+1} - a_{i}$. Condition $\bfa\in W$ implies $d_{i} > 0$ for all $i=1,2, \dots,n-1$.
We take  $s={\max\{-b_i,0\}}/{\min\{d_i\}}$ and $t>s$.
Set $z_i = y_i + t a_i $, then we get that $\bfz \in W$, because
$$z_{i+1} - z_i = y_{i+1} + t a_{i+1} - y_i - t a_i = b_i +t d_i > b_i + s d_i \geq b_i + \max\{-b_i,0\} \geq 0.$$
Thus for $t>s$ we have $ \bfy = \bfz - t \bfa $,
where $\bfz\in W$, and so $\bfy\in \{W-ta\}$ for all  $t>s$.
\halmos
\end{proof}

\medskip\noindent
{\bf Example 3} $(a_1 < a_2 < \dots < a_n)$ 
 This is the case of non-colliding drifts.
Here $q = q' = n$, $m_1 = m'_{1}= 1, \dots , m_n = m'_{n} = n$,
$f_{m_1}= a_1, \dots, f_{m_n} = a_n$.
Using Proposition \ref{GBM.MainProp} we have
\begin{eqnarray*}
\Prob_{\bfx}(\tau>t) &=&
(2\pi)^{-n/2}
e^{-<\bfx,\bfa>}
\det \left[e^{x_k a_j}\right]
\int_{W-\bfa\sqrt{t}}
e^{-\frac{1}{2}|\bfz|^2} \ud \bfz \ (1 + o(1)).
\end{eqnarray*}
\\
By Lemma \ref{lemmaW->Rn} we have that
\begin{eqnarray*}
\lim_{t \rightarrow \infty}
\int_{W-\bfa\sqrt{t}}
e^{-\frac{1}{2}|\bfz|^2} \ud \bfz =
\int_{\bbR^n}
e^{-\frac{1}{2}|\bfz|^2} \ud \bfz = (2\pi)^{n/2}.
\end{eqnarray*}
\\
Finally we write
$$
\lim_{t \rightarrow \infty} \Prob_{\bfx}(\tau>t) =
e^{-<\bfx,\bfa>} \det \left[e^{x_k a_j}\right].
$$
This result was derived earlier by Biane {\it et al} \cite{bianeetal}

\medskip\noindent
{\bf Example 4} Case when $q=q'=1$.  This is the case of a one irreducible drift vector.
Here $m_1 = m'_{1}= n$,
$f_1 = f_2 \dots =  f_n = \bar{a}_{[1;n]} = \frac{a_1 + \dots + a_n}{n}$.
Using Proposition \ref{GBM.MainProp} we have
\begin{eqnarray*}
\Prob_{\bfx}(\tau>t)
&=&
C h(\bfx) t^{- \alpha} e^{-\gamma t} (1 + o(1)),
\end{eqnarray*}
where
\begin{eqnarray*}
\gamma = \frac{1}{2}  \left( \frac{1}{n} \sum_{0<u<v\leq n} (a_u - a_v)^2 \right),
\end{eqnarray*}

\begin{eqnarray*}
\alpha = \frac{(n-1)(n+1)}{2},
\end{eqnarray*}
\begin{eqnarray*}
h(\bfx)&=&
e^{-<\bfx,\bfa>} \det\left[ e^{x_i f_j} x_i^{j-1} \right] ,
\end{eqnarray*}
\begin{eqnarray*}
C &=&
(2\pi)^{-n/2} \sqrt{2\pi n}c_n
  \\
&&\times
\idotsint\limits_{\xi_i>0 : i = 1,2, \dots , n-1 }
e^{-\frac{1}{2}\sum_{l=1}^q \left( \frac{2}{\nu_l} \sum_{m_{l-1}<u<v\leq m_l} (\xi_u + \dots + \xi_{v-1})(a_u - a_v) \right)}
\nonumber \\
&&\hspace{7cm}\times  H(\bfxi_{[1;n-1]})
\prod_{ i = 1}^{n-1}\ud \xi_i\;.
\nonumber
\end{eqnarray*}

\bigskip
We now analyze a remaining situation for $n = 3$.  

\medskip\noindent
{\bf Example 5}
$(a_1 > a_2$ and $\frac{a_1 + a_2}{2} < a_3)$. This is the case of two subsequences.
Thus  $q = 2, \ q' = 2$ and $m_1 = m'_{1} = 2, \ m_2 = m'_{2} = 3$.
By Theorem \ref{GBM.MainTh} we have
\begin{eqnarray*}
\Prob_{\bfx}(\tau>t)
&=&
C h(\bfx)
e^{-\frac{t}{4}(a_2 - a_1)^2 }    t^{-\frac{3}{2} },
\end{eqnarray*}
where
\begin{eqnarray*}
&\gamma & = \ \frac{(a_2 - a_1)^2}{4}, \\
&\alpha& = \ \frac{3}{2}, \\
&h(\bfx)& = \ e^{-<\bfx,\bfa>}
\left|%
\begin{array}{ccc}
  e^{x_1 \frac{a_1 + a_2}{2}} & e^{x_1 \frac{a_1 + a_2}{2}} x_1 & e^{x_1 a_3} \\
  e^{x_2 \frac{a_1 + a_2}{2}} & e^{x_2 \frac{a_1 + a_2}{2}} x_2 & e^{x_2 a_3} \\
  e^{x_3 \frac{a_1 + a_2}{2}} & e^{x_3 \frac{a_1 + a_2}{2}} x_3 & e^{x_3 a_3} \\
\end{array}%
\right|
,\\
&C& = \  (2\pi)^{-3/2} \sqrt{2\pi 3}  \frac{1}{(a_1 - a_2)^2} \sqrt{3 \pi}.
\end{eqnarray*}

\section{Auxiliary results.}\label{s.proof}
For the proof we need a set of lemmas and propositions, presented in subsections below.

\subsection{Useful lemmas.}
We need a few technical lemmas, which we state without proofs.
\begin{Lemma}\label{lemma.y^2_cz1}
For $\bfa \in \bbR^m$
$$
\sum_{i=1}^m \left(\bar{a}_{[1;m]} - a_i \right)^2 = \frac{1}{m} \sum_{1\leq u<v \leq m} (a_u - a_v)^2.
$$
\end{Lemma}

\begin{Lemma}\label{lemma.y^2_cz2}
For $\bfa , \bfz \in \bbR^m$
$$
\sum_{i=1}^m  z_i \left( \bar{a}_{[1;m]} - a_i \right) = \frac{2}{m} \sum_{u<v} (z_v-z_u)(a_u - a_v).
$$
\end{Lemma}

The proof of the following lemma follows easily from Lemmas \ref{lemma.y^2_cz1} and \ref{lemma.y^2_cz2}.
\begin{Lemma}\label{lemma.y^2}
For $\bfa, \bff \in R^n$ such that $\bff$ is is a vector obtained from the stable partition of $\bfa$, and $\bfz \in R^n$, we have
\begin{eqnarray*}
|\bff\sqrt{t} - \bfa\sqrt{t} + \bfz|^2 = |\bfz|^2 + \sum_{l=1}^q  &\bigg(& \frac{t}{\nu_l} \sum_{m_{l-1}<u<v\leq m_l} (a_u - a_v)^2  \\
&+&
\frac{2\sqrt{t}}{\nu_l} \sum_{m_{l-1}<u<v\leq m_l} (z_v-z_u)(a_u - a_v) {\bigg)}.
\end{eqnarray*}

\end{Lemma}

\begin{Lemma}\label{MatrixIdent} 
$$
\bfA = \left(%
\begin{array}{cccccc}
  -1 & 1 & 0 & \dots& 0 & 0 \\
  0 & -1 & 1 & \dots& 0 & 0 \\
  \vdots & \vdots & \vdots & \ddots & \vdots &\vdots\\
  0 & 0 & 0 &\dots & -1 & 1 \\
  1 & 1 & 1 & \dots & 1 &1\\
\end{array}%
\right)\;,
$$
then
$$
(\bfA^{-1})^T \bfA^{-1} = \frac{1}{n}
 \left(%
\begin{array}{cccccc}
  n-1 & &  & & &  \\
  n-2 & 2(n-2) & & & & \\
  n-3 & 2(n-3) &3(n-3) & & & \\
  \vdots & \vdots & \vdots & \ddots &  &\\
  1 & 2 & 3 &\dots & n-1 &  \\
  0 & 0 & 0 & \dots & 0 &1\\
\end{array}%
\right)\;.
$$
Note that $(\bfA^{-1})^T \bfA^{-1}$ is symmetric.
\end{Lemma}

\bigskip

By Proposition \ref{GBM.Prop.assym.po.zamianie} we have
\begin{eqnarray*}
\Prob_{\bfx}(\tau>t)
&=& (2\pi)^{-n/2}
e^{-<\bfa,\bfx> -||\bfx||^2/2t}
\int_W e^{-||\bfy - \bfa \sqrt{t}||^2 /2}
\det(e^{x_iy_j/\sqrt{t}})\ud\bfy \;.
\end{eqnarray*}
We now introduce new variable $\bfz$ by
$$
\bfy = \bff \sqrt{t} + \bfz \label{podstawienie},
$$
where $\bff = (f_1, \dots , f_n)$ is a vector obtained from the stable partition of $\bfa$.

Finally we rewrite formula \refs{GBM.Prop.assym.po.zamianie.wzor} in new variables by the use of Lemma \ref{lemma.y^2}:
\begin{Lemma}
\begin{eqnarray} \label{GBM.Wzor.Po.Podstawieniu}
\lefteqn{\Prob_{\bfx}(\tau>t)=(2\pi)^{-n/2}
e^{-<\bfa,\bfx> -||\bfx||^2/2t}e^{-\gamma t}} \nonumber\\[.5cm]
&&\int_{W-\bff\sqrt{t}}
e^{-\frac{|\bfz|^2}{2} }
 e^{-\frac{1}{2} \sum_{l=1}^q \frac{2\sqrt{t}}{\nu_l} \sum_{m_{l-1}<u<v\leq m_l} (z_v-z_u)(a_u - a_v)  }
\det\left[e^{x_i(z_j/\sqrt{t} + f_j) }\right]
\ud \bfz.
\end{eqnarray}
\end{Lemma}

\subsection{Asymptotic behavior of determinant. }

The following lemma is an extension of Lemma 2 from Pucha\l a \cite{puchala}\,.

 We define functions
$$
g_{\bfk}(\bfz) = \frac{\det[z_i^{k_j}]}{\det[z_i ^{j-1}]},
$$
 for $\bfk = (k_1 , \dots , k_n) \in \bbZ^n$ and $0 \leq k_1 < \dots < k_n$
Functions $g$ corresponds to 
Schur functions $g_{\bfk} = s_{\bfk - (0,1, \dots , n)}$; see e.g. Macdonald \cite{macdonald}, Ch. 1.3.
\begin{Lemma}\label{GBM.DetLemma}
Let $k_0=\sum_{j=1}^{q'} \binom{\nu'_j}{2}$
\begin{eqnarray} \label{GBM.DetExpansion}
\det\left[e^{x_i (z_j/\sqrt{t} + f_j)} \right] &=&
\sum_{k = k_0}^{\infty }
t^{-k/2} T_k \; \; ,
\end{eqnarray}
where
\begin{eqnarray*}
T_k &=&
\prod_{j=1}^{q'}\Delta \left(\bfz_{(m'_{j-1},m'_l]} \right)\\
&&\times
\sum_{{k_1 + \dots + k_n = k  \atop k_1< \dots <k_{m'_1};}
\atop{{\vdots \dots \vdots}\atop{k_{m'_{q'-1}+1}} <\dots < k_{m'_{q'}}}}
\frac{g_{\bfk_{(m'_0, m'_1]}}(\bfz_{(m'_0, m'_1]}) }{k_1 ! \dots k_{m'_1} !}
\dots
\frac{g_{\bfk_{(m'_{q'-1}, m'_{q'}]}}(\bfz_{(m'_{q' -1}, m'_{q'}]})} {k_{m'_{q'-1}+1} ! \dots k_{m'_{q'}} !}
\det\left[ e^{x_i f_j} x_i^{k_j}  \right]\;.
\end{eqnarray*}
In particular
as $t \rightarrow \infty$
\begin{eqnarray*}
\lefteqn{
\det\left[e^{x_i (z_j/\sqrt{t} + f_j)} \right] =}\\
&=& t^{-\frac{1}{2}\sum_{j=1}^{q'} \binom{\nu'_j}{2}}
\prod_{j=1}^{q'}c_{\nu'_j}\Delta(\bfz_{(m'_{j-1},m'_l]})\\
&\times&
\det \left[e^{x_k f_j} x_k^{\sum_l^{q'} (j-m_{i_{l-1}} - 1)  \ind_{\{m_{i_{l-1}} < j \leq m_{i_l}\}}} \right] (1 + o(1)).
\end{eqnarray*}
\end{Lemma}

\proof
By $S_n$ we denote the group of permutations on $n$-set. We write
\begin{eqnarray*}
\lefteqn{
    \det\left[e^{x_i (z_j/\sqrt{t} + f_j)} \right] =}\\
&=&
\sum_{\sigma \in S_n} (-1)^{\sigma} e^{\sum x_i f_{\sigma(i)}} e^{\sum x_i z_{\sigma(i)}/ \sqrt{t}} \\
&=&
\sum_{\sigma \in S_n} (-1)^{\sigma} e^{\sum x_i f_{\sigma(i)}}\sum_{k=0}^{\infty} t^{-k/2}({ x_1 z_{\sigma(1)} + \dots + x_n z_{\sigma (n)} })^k/k! \\
&=&
\sum_{k=0}^{\infty} \frac{t^{-k/2}}{k!}\sum_{\sigma \in S_n} (-1)^{\sigma} e^{\sum x_i f_{\sigma(i)}}({ x_1 z_{\sigma(1)} + \dots + x_n z_{\sigma (n)} })^k .
\end{eqnarray*}
Now the coefficient at $t^{-k/2}$ is equal to
\begin{eqnarray*}
T_k &=& T_k(\bfz)= \frac{1}{k!} \sum_{\sigma \in S_n} (-1)^{\sigma} e^{\sum x_i f_{\sigma(i)}}
({ x_1 z_{\sigma(1)} + \dots + x_n z_{\sigma (n)} })^k \\
&=&
\frac{1}{k!}\sum_{\sigma \in S_n} (-1)^{\sigma} e^{\sum x_i f_{\sigma(i)}}
\sum_{k_1 + \dots + k_n = k} \frac{k!}{k_1 ! \dots k_n!}
(x_1 z_{\sigma(1)})^{k_{\sigma(1)}} \dots (x_n z_{\sigma(n)})^{k_{\sigma(n)}} \\
&=&
\sum_{k_1 + \dots + k_n = k} \frac{1}{k_1 ! \dots k_n!} \sum_{\sigma \in S_n} (-1)^{\sigma}
e^{\sum x_i f_{\sigma(i)}}(x_1 z_{\sigma(1)})^{k_{\sigma(1)}} \dots (x_n z_{\sigma(n)})^{k_{\sigma(n)}} \\
&=&
\sum_{k_1 + \dots + k_n = k} \frac{\bfz^{\bfk} }{k_1 ! \dots k_n!}  \det[e^{x_i f_j} x_i ^{k_j}] .
\end{eqnarray*}
Recall that 
$$
f_1 = \dots = f_{m'_{1}} < f_{m'_{1} + 1} = \dots = f_{m'_{2}} < \dots < f_{m'_{q'-1}+1} = \dots =f_{m'_{{q'}}}.
$$
If $k_i = k_j$ and $f_i = f_j$, then the determinant $\det\left[e^{x_i f_j} x_i ^{k_j} \right]$ is $0$.
Thus we have non-zero determinant if  $k_i$ are different for those $i$ such that $f_i$ are equal.
Thus  index $k$ such that $T_k$ is non-zero must be at least
$$
k = \sum_{j=1}^n k_j \geq k_0=\sum_{j=1}^{q'} \binom{\nu'_j}{2},
$$
Moreover
we get all nonzero  $\det\left[e^{x_i f_j} x_i ^{k_j} \right]$
putting in each subsequence
$$(\bfk_{(m'_0,m'_1]},\ldots,\bfk_{(m'_{q'-1},m'_{q'}]})\;,$$
all possible permutations of strictly ordered numbers from $\bbZ_+$ such that all sum up to $k$.
Thus we have
\begin{eqnarray*}
T_k &=& \sum_{{k_1 + \dots + k_n = k  \atop k_1< \dots <k_{m'_1};}
\atop{{\vdots \dots \vdots}\atop{k_{m'_{q'-1}+1}} <\dots < k_{m'_{q'}}}}
\sum_{\sigma_1 \in S_{\nu'_1}} \dots \sum_{\sigma_{q'} \in S_{\nu'_{q'}}}
\frac{\bfz_{(m'_0,m'_1]}^{\sigma_1(\bfk_{(m'_0,m'_1]})}  }{k_1 ! \dots k_{m'_1} !}
\dots
\frac{\bfz_{(m'_{q' -1},m'_{q'}]}^{\sigma_1(\bfk_{(m'_{q' -1},m'_{q'}]})}  }{k_{m'_{q'-1}+1} ! \dots k_{m'_{q'}} !} \\
&&\quad \times
\det\left[ e^{x_i f_j} x_i ^{\sum_{l=1}^{q'} \sigma_l (k_j) 1_{m'_{l-1}< j \leq m'_l}}
       \right]
\end{eqnarray*}
Again we notice that permutations in the determinant influence only by the change of sign.
These signs and sums over the group of permutations form determinants, thus we have
\begin{eqnarray*}
T_k
&=& \sum_{{k_1 + \dots + k_n = k  \atop k_1< \dots <k_{m'_1};}
\atop{{\vdots \dots \vdots}\atop{k_{m'_{q'-1}+1}} <\dots < k_{m'_{q'}}}}
\frac{\det\left[\left\{z_i^{k_j} \right\}_{i,j=1}^{m'_1}\right] }{k_1 ! \dots k_{m'_1} !}
\dots
\frac{\det\left[\left\{z_i^{k_j} \right\}_{i,j=m'_{q'-1}+1}^{m'_{q'}}\right] } {k_{m'_{q'-1}+1} ! \dots k_{m'_{q'}} !}
\det\left[ e^{x_i f_j} x_i^{k_j}  \right]\;.
\end{eqnarray*}
\halmos

\medskip\noindent{\bf Remark}.
Using Itzykson--Zuber integral (see e.g. \cite{itzyksonzuber})
we can write
$$\frac{\det\left[e^{x_i (z_j/\sqrt{t} + f_j)} \right]}
{\Van(\bfx)\Van(\bfz/\sqrt{t}+\bff)}=c_n\int_{\bfU(n)}
e^{{\rm Tr}  \diag(\bfx)U\diag(\bfz/\sqrt{t}+\bff)U^{*}}\mu(\ud U)\;,
$$
where
$\mu(\ud U)$ is (normalized) Haar measure on the unitary group $\bfU(n)$.
Now letting
$t\to\infty$,
\begin{eqnarray*}
\int_{\bfU(n)} e^{{\rm Tr}
\diag(\bfx)U\diag(\bfz/\sqrt{t}+\bff)U^{*}}\mu(\ud U)\;&\to&
\int_{\bfU(n)} e^{{\rm Tr}( \diag(\bfx)U\diag(\bff)U^{*})}\mu(\ud U)\\
&=&\frac{\det\left[e^{x_i f_j} \right]} {\Van(\bfx)\Van(\bff)}
\end{eqnarray*}
and $$\Van(\bfz/\sqrt{t}+\bff)=t^{-\sum_{i=1}^{q'}{\nu'_i\choose 2}}
\prod_{i=1}^{q'}\Van(\bfz_{(m'_{i-1};m'_i]})\prod_{1\le k<l\le n} (f^l-f^k)^{\nu'_k\nu'_l}
 (1+o(1))\;.
$$
Hence, as $t\to\infty$
\begin{eqnarray*}
\lefteqn{\det\left[e^{x_i (z_j/\sqrt{t} + f_j)} \right]}\\
&&\to t^{-\sum_{i=1}^{q'}{\nu'_i\choose 2}}
c_n\prod_{i=1}^{q'}\Van(\bfz_{(m'_{i-1};m'_i]})
\prod_{1\le k<l\le n} (f^l-f^k)^{\nu'_k\nu'_l}
\frac{\det\left[e^{x_i f_j} \right]}
{\Van(\bff)}\;.
\end{eqnarray*}
This is a less detailed version of the formula from Lemma \ref{GBM.DetLemma}.

\section{Proof of the Theorem.}
Using \refs{GBM.DetExpansion} and formula \refs{GBM.Wzor.Po.Podstawieniu} we write
\begin{eqnarray} \label{GBM.caloscZSuma}
\lefteqn{\Prob_{\bfx}(\tau>t) =} \nonumber \\
&=&
(2\pi)^{-n/2}
e^{ -||\bfx||^2/2t}e^{-<\bfx,\bfa>}e^{-\gamma t}\nonumber\\
&&\times 
\sum_{k = k_0}^{\infty }
\int_{W-\bff\sqrt{t}}
e^{-\frac{1}{2}|\bfz|^2} 
e^{-\frac{1^2}{2}
\sum_{l=1}^q \left( \frac{2\sqrt{t}}{\nu_l } \sum_{m_{l-1}<u<v\leq m_l} (z_u-z_v)(a_v - a_u) \right)}
t^{-k/2} T_k(\bfz) \; \; \ud \bfz \; \; , \label{firstterm}
\end{eqnarray}
First we will analyze above expression by taking only the first term 
in the
sum \refs{firstterm}, 
and then we show that it gives the right asymptotic.
Thus
the first term equals to
\begin{eqnarray*}
&&(2\pi)^{-n/2}
e^{ -||\bfx||^2/2t}e^{-\gamma t}
\int_{W-\bff\sqrt{t}}
e^{-\frac{1}{2}|\bfz|^2}\\
&&\hspace{1.0cm}\times
e^{-\frac{1^2}{2}
\sum_{l=1}^q \left( \frac{2\sqrt{t}}{\nu_l } \sum_{m_{l-1}<u<v\leq m_l} (z_u-z_v)(a_v - a_u) \right)}\\
&&\hspace{1.0cm}\times
e^{-<\bfx,\bfa>}\prod_{j=1}^{q'}c_{\nu'_j}
\Delta(\bfz_{(m'_{j-1};m'_j]})\\
&&\hspace{0.5cm}\times
\det \left[e^{x_k f_j} x_k^{\sum_l^{q'} (j-m_{i_{l-1}} - 1)  \ind_{\{m_{i_{l-1}} < j \leq m_{i_l} \}}} \right]
t^{-\frac{1}{2} k_0  }  \ud \bfz \\
&=&
(2\pi)^{-n/2}
e^{ -||\bfx||^2/2t}e^{-\gamma t} \\
&&\hspace{1.0cm}\times
e^{-<\bfx,\bfa>}
\det \left[e^{x_k f_j} x_k^{\sum_l^{q'} (j-m_{i_{l-1}} - 1)  \ind_{\{m_{i_{l-1}} < j \leq m_{i_l} \}}} \right]
t^{-\frac{1}{2} k_0 }\\
&&\hspace{1.0cm}\times \prod_{j=1}^{q'}c_{\nu'_j}
  \,I(\bfa,t) ,
\end{eqnarray*}
where $I(\bfa,t)$ was introduced in \refs{GBM.I()}.

\subsection{Asymptotic behavior of integral.}\label{ss.asymptotic}
If $\bfs = \bfA \bfz$, where
$$
\bfA = \left(%
\begin{array}{cccccc}
  -1 & 1 & 0 & \dots& 0 & 0 \\
  0 & -1 & 1 & \dots& 0 & 0 \\
  \vdots & \vdots & \vdots & \ddots & \vdots &\vdots\\
  0 & 0 & 0 &\dots & -1 & 1 \\
  1 & 1 & 1 & \dots & 1 &1\\
\end{array},%
\right)
$$
than  $z_u - z_v = s_v  + s_{v+1} + \dots + s_{u-1}$ and
$$
|\bfz|^2 = \bfz^T \bfz = (\bfA^{-1}\bfs)^T (\bfA^{-1}\bfs) = \bfs^T (\bfA^{-1})^T \bfA^{-1}\bfs.
$$
Hence by  Lemma \ref{MatrixIdent} we have
$$
|\bfz|^2 = \frac{1}{n}s_n^2 + \bfs_{(n)}^T ((\bfA^{-1})^T \bfA^{-1})_{(n)} \bfs_{(n)},
$$
where $\bfs_{(n)}$ is obtained from $\bfs$ by deleting the $n^{\text{th}}$ coordinate
and  $\bfA_{(n)}$ is  matrix $\bfA$ without $n^{\text{th}}$ row and $n^{\text{th}}$ column.

After substitution $\bfs = \bfA \bfz$, integral $I(\bfa,t)$ is
\begin{eqnarray}
I(\bfa , t) &=&
 \idotsint\limits_{
{s_i>(f_{i}-f_{i+1})\sqrt{t},}
\atop { \text{for }
i=1,\ldots,n-1}}\quad
\int_{s_n\in R} e^{-\frac{1}{2}
(\frac{s_n^2}{n} + \bfs_{(n)}^T ((\bfA^{-1})^T \bfA^{-1})_{(n)} \bfs_{(n)})}\nonumber\\[0.5cm]
&&\hspace{1cm}
\times
e^{-\frac{1}{2}\sum_{l=1}^q \left( \frac{2\sqrt{t}}{\nu_l}
\sum_{m_{l-1}<u<v\leq m_l} (s_u + \dots + s_{v-1})(a_u - a_v) \right)}\label{secondexp} \\[0.5cm]
&&\hspace{1cm}  \times
\prod_{k=1}^{q'} H(\bfs_{(m'_{k-1} ; m'_{k})}\,\ud\bfs_{(n)} \,\ud s_n\nonumber  \\[0.5cm]
&=&
\sqrt{2\pi n}
\idotsint\limits_{s_i>(f_{i}-f_{i+1})\sqrt{t}, \atop
\text{for } i=1,\ldots,n-1}
 e^{-\frac{1}{2}
( \bfs_{(n)}^T ((\bfA^{-1})^T \bfA^{-1})_{(n)} \bfs_{(n)})}\nonumber\\[0.5cm]
&& \hspace{1cm}  \times  e^{-\frac{1}{2}\sum_{l=1}^q \left( \frac{2\sqrt{t}}{\nu_l } \sum_{m_{l-1}<u<v\leq m_l} (s_u + \dots + s_{v-1})(a_u - a_v) \right)} \\
&& \hspace{1cm}  \times\prod_{k=1}^{q'} H(\bfs_{(m'_{k-1}, m'_{k})})\,\ud\bfs_{(n)} . \nonumber
\end{eqnarray}
It is important to notice
that the second exponent in integral $I(\bfa , t )$ in \refs{secondexp} depends only on those $s_i$, where $i \notin \{m_1, \dots m_q \}$. We also see that if $m_{i-1} < k < m_i$, then the coefficient at $s_k$ in \refs{secondexp}
 is
$$
-\frac{\sqrt{t}}{m_i - m_{i-1}} (m_{i}- k) (k - m_{i-1})
\left(
\frac{a_{m_{i-1}+1} + \dots + a_i}{i-m_{i-1}} - \frac{a_{i+1} + \dots + a_{m_{i}}}{m_{i}-i}
\right)
$$
and it is strictly negative by the definition of the stable partition.
Note also that polynomials $H$ in integral $I(\bfa , t)$ depends only on $s_j$, where $j \notin \{m'_{1}, \dots , m'_{{q'}}\}$.

We now introduce new variables $\bfxi=(\xi_1,\ldots,\xi_{n-1})$ by
\begin{eqnarray} \label{podstawienie.xi}
\xi_j = \left\{
\begin{array}{ll}
\sqrt{t} s_j, & \text{ for } j \neq m_i,\  j=1,\ldots,n-1, \quad i = 1, \dots , q-1\\ [0.4cm]
 s_j, & \text{ for } j = m_i, \ j=1,\ldots,n-1, \quad i = 1, \dots , q-1.
\end{array}
\right.
\end{eqnarray}
We define function $K$ by $K\left(\bfxi_{(m'_{k-1}, m'_{k})},t\right)= H\left(\bfs_{(m'_{k-1}, m'_{k})}\right)$.

Consider now $H\left(\bfs_{(1;m'_{1})}\right)$.
Since $\bfm'$ is a subsequence of $\bfm$, we recall that $i_1$ is such that $m_{i_1} = m'_{1}$.
Similarly are defined $i_1,\ldots,i_{q'}$.
We now factorize $H\left(\bfs_{(1;m'_{1})}\right)$ into parts in which there in none of $m_i$,
 where is exactly one $m_i$, exactly two and so on.
Thus
\begin{eqnarray*}
H \left(\bfs_{(m_0;m'_{1})} \right)&=&  \prod_{k=1}^{i_1}H \left(\bfs_{(m_{k-1}; m_{k})} \right)\\
& &
\times
\prod_{k=1}^{i_1 - 1}
\prod_{
    {m_{k-1} < i \leq m_{k} }
    \atop {m_{k}<j \leq m_{k+1}}
}
(s_i + \dots + s_{j-1}) \\
& & \times
\prod_{k=1}^{i_1 - 2}
\prod_{
    {m_{k-1} < i \leq m_{k}}
    \atop {m_{k+1}<j \leq m_{k+2}}
}
(s_i + \dots + s_{j-1}) \\
& &\vdots\\
& &\times
\prod_{
    {m_{0} < i \leq m_{1}}
    \atop {m_{i_1 - 1}<j \leq m_{i_1}}
}
(s_i + \dots + s_{j-1}).
\end{eqnarray*}
We make analogous factorization for other $H(\bfs_{(m_{k-1};m_{k})})$.

\begin{Lemma}\label{asymptoticH} As $t\to\infty$
\begin{eqnarray*}
\prod_{k=1}^{q'}K(\bfxi_{(m'_{k-1}, m'_{k})},t)
& = & 
t^{- \frac{1}{2} \sum_{i=1}^q \binom{\nu_i }{2}} \prod_{i=1}^q H(\bfxi_{(m_{i-1};m_i)}) \\
& & \times 
\prod_{k=0}^{q-1}
\prod_{
    { i : \{i,i+1,\dots , i+k\} }
    \atop{ \in \{1, \dots , q \} \setminus \{ {i_1}, \dots , {i_{q'}} \}
    }}
\left(\sum_{j=0}^{k} \xi_{m_{i+j}} \right)^{\nu_i \nu_{i+k+1}} (1 + o(1))
\end{eqnarray*}

\end{Lemma}

\proof
After the substitution we get
\begin{eqnarray*}
K(\bfxi_{(m'_{k-1}, m'_k)},t)
&=&
\prod_{k=1}^{i_1}H \left(\bfxi_{(m_{k-1}; m_{k}) }/ \sqrt{t} \right)\\
&&\times
\prod_{k=1}^{i_1-1}
\prod_{
    {m_{k-1} < i \leq m_{k}}
    \atop {m_{k}<j \leq m_{k+1} }
}
\left(\sum_{r=i, r\neq m_k}^{j-1} \xi_r/\sqrt{t}  + \xi_{m_k} \right) \\
&&\times
\prod_{k=1}^{i_1 - 2}
\prod_{
    {m_{k-1} < i \leq m_{k}}
    \atop {m_{k+1}<j \leq m_{k+2}}
}
\left(\sum_{r=i, r\notin \{m_k,m_{k+1}\}}^{j-1} \xi_r/\sqrt{t}  + \xi_{m_k} + \xi_{m_{k+1}}
\right) \\
&&\vdots\\
&&\times
\prod_{
    {m_{0} < i \leq m_{1}}
     \atop {m_{i_1 -1}<j \leq m_{ i_1 } }
}
\left(\sum_{r=i, r\notin \{m_1,\dots , m_{i_1-1}\}}^{j-1} \xi_r/\sqrt{t}  + \sum_{k=1}^{i_1 - 1} \xi_{m_k}
\right).
\end{eqnarray*}
It is not difficult to see that asymptotic behavior of the above expression is

\begin{eqnarray*}
K(\bfxi_{(m'_{k-1}, m'_k)},t)
&=& t^{-\frac{1}{2}\sum_{l=1}^{i_1 } \binom{\nu_i}{2}}
\prod_{k=1}^{i_1}H \left(\bfxi_{(m_{k-1}; m_{k}) } \right)\\
& &\times
\prod_{k=1}^{i_1 - 1}
(\xi_{m_k})^{\nu_k \nu_{k+1}} \\
&&
\times
\prod_{k=1}^{i_1-2}
(\xi_{m_k} + \xi_{m_{k+1}})^{\nu_k \nu_{k+2}} \\
& &  \vdots\\
& &\times
\left( \sum_{k=1}^{i_1 - 1} \xi_{m_k} \right)^{\nu_1 \nu_{i_1}} \, (1 + o(1)) .
\end{eqnarray*}

In result the whole polynomial is asymptotically

\begin{eqnarray*}
\prod_{k=1}^{q'} K(\bfxi_{(m'_{k-1}; m_{k})},t) &=&
t^{- \frac{1}{2} \sum_{i=1}^q \binom{\nu_i}{2}}
    \prod_{i=1}^q H \left(\bfxi_{(m_{i-1};m_i)} \right)
\\
&& \times
\prod_{k=0}^{q-1} \prod_{
 i : \{i,i+1,\dots , i+k\}
 \atop \in \{1, \dots , q \} \setminus \{ {i_1}, \dots , {i_{q'}} \}}
    \left(
        \sum_{j=0}^{k} \xi_{m_{i+j}}
    \right)^{\nu_i \nu_{i+k+1}} (1 + o(1)).
\end{eqnarray*}
\halmos

For substitution \refs{podstawienie.xi}, we have $\ud \bfs_{(n)} = t^{-(n-q)/2} \ud \bfxi$.
Note that $f_{k+1} = f_{k}$ for   $k \neq m_i$, and hence   the integration on the $k^{\text{th}}$ coordinate
starts from $0$. On the other hand if $k = m_i$ for some $i$, and $k \neq m'_{j}$ for every $j$,
then we also have $f_{k+1} = f_{k}$ and therefore  the integration starts from $0$.
Finally if $k = m_{i_j}$ for some $j$, then $f_{k+1} > f_{k}$ and the integrations starts from $(f_{k} - f_{k+1})\sqrt{t}$.
Hence we have after the substitution
%
%
%
\begin{eqnarray*}
I(\bfa,t)&=&
t^{-(n-q)/2} \sqrt{2\pi n}
\idotsint\limits_{\xi_j>(f_{j}-f_{j+1})\sqrt{t}, \atop \text{for } i=1,\ldots,n-1} e^{-\frac{1}{2} (
\sum_{k,l\in \{m_1, \dots, m_q\}} S_{kl} \xi_k \xi_l  )}\\
&&\times
e^{-\frac{1}{2} (
\sum_{k,l\notin \{m_1, \dots, m_q\}} S_{kl} \xi_k \xi_l/t +
2 \sum_{k\in \{m_1, \dots, m_q\},l\notin \{m_1, \dots, m_q\}} S_{kl} \xi_k \xi_l/\sqrt{t}  )}\\[0.4cm]
&&\times
e^{-\frac{1}{2}\sum_{l=1}^q \left( \frac{2}{\nu_l} \sum_{m_{l-1}<u<v\leq m_l} (\xi_u + \dots + \xi_{v-1})(a_u - a_v) \right)} \\[0.4cm]
&&\times \prod_{k=1}^{q'} K(\bfxi_{(m'_{k-1}, m'_{k})}, t )
\ud \bfxi\;.
\end{eqnarray*}


So we can clearly see that $\prod_{k=1}^{q'} K(\bfxi_{(m'_{k-1}, m'_{k})}, t )$
 depends only on $\xi_i$'s such that $i \notin  \{ m_{l_1}, \dots , m_{l_{q'}} \}$
and it can be factorized   into a part which depends only on
$i \notin \{ m_1 , \dots , m_q\}$ and a part that depends on
$i \in \{m_1, \dots , m_q \} \setminus \{ m_{l_1}, \dots , m_{l_{q'}} \}$.
Thus finally we can write
\begin{eqnarray*}
I(\bfa , t)&=&
t^{-(n-q)/2} \sqrt{2\pi n}
\idotsint\limits_{\xi_i>(f_{i}-f_{i+1})\sqrt{t} \atop \text{for } i = 1 , \dots , n-1}
e^{-\frac{1}{2} ( \sum_{k,l\in \{m_1, \dots, m_q\}} S_{kl} \xi_k \xi_l  )}\\
&&\times
e^{-\frac{1}{2} (
\sum_{k,l\notin \{m_1, \dots, m_q\}} S_{kl} \xi_k \xi_l/t +
2 \sum_{k\in \{m_1, \dots, m_q\},l\notin \{m_1, \dots, m_q\}} S_{kl} \xi_k \xi_l/\sqrt{t}  )}\\
&&\times
e^{-\frac{1}{2}\sum_{l=1}^q \left( \frac{2}{\nu_l} \sum_{m_{l-1}<u<v\leq m_l} (\xi_u + \dots + \xi_{v-1})(a_u - a_v) \right)} \\
&&\times
t^{- \frac{1}{2} \sum_{i=1}^q \binom{\nu_i}{2}} \prod_{i=1}^q H(\bfxi_{(m_{i-1};m_i)},t) \\
&&\times
\prod_{k=0}^{q-1} \prod_{i : \{i,i+1,\dots , i+k\} \atop \in \{1, \dots , q \} \setminus \{ {i_1}, \dots , {i_{q'}} \}}
\left(\sum_{j=0}^{k} \xi_{m_{i+j}}\right)^{\nu_{i}\nu_{i+k+1}}
\ud \bfxi \ (1 + o(1))
\end{eqnarray*}
Hence
\begin{eqnarray}
I(\bfa , t) &=&
t^{-(n-q)/2} t^{- \frac{1}{2} \sum_{i=1}^q \binom{\nu_i}{2}}
\sqrt{2\pi n} \nonumber\\
&&\times
\idotsint\limits_{\xi_i>0 : i = 1 , \dots , n-1 \atop i \notin \{m_1,\dots , m_q\} }
e^{-\frac{1}{2}\sum_{l=1}^q \left( \frac{2}{\nu_l} \sum_{m_{l-1}<u<v\leq m_l}
(\xi_u + \dots + \xi_{v-1})(a_u - a_v) \right)}\nonumber\\
&&\times
\prod_{i=1}^q H(\bfxi_{(m_{i-1};m_{i})})
\prod_{i \notin \{m_1,\dots , m_q\} } \ud \xi_i\nonumber\\
&&\times
\idotsint\limits_{\xi_i > 0: i \in \{m_1, \dots , m_q \} \setminus \{ m_{l_1}, \dots , m_{l_{q'}} \} }
\idotsint\limits_{\xi_i > - \infty : i \in \{ m_{l_1}, \dots , m_{l_{q'}} \} }
e^{-\frac{1}{2}
\left( \sum_{k,l\in \{m_1, \dots, m_q\}} S_{kl} \xi_k \xi_l
\right)}\nonumber\\
&&\times
\prod_{k=0}^{q-1}
\prod_{i : \{i,i+1,\dots , i+k\} \atop \in \{1, \dots , q \} \setminus \{ {l_1}, \dots , {l_{q'}} \}}
\left(\sum_{j=0}^{k} \xi_{m_{i+j}}
\right)^{\nu_{i}\nu_{i+k+1}}
\prod_{i \in \{m_1,\dots , m_q\}} \ud \xi_i \, \, (1 + o(1))\;.\label{finalI}
\end{eqnarray}
Concluding we have
\begin{eqnarray*}
I(\bfa,t) &=&
C_1 t^{-(n-q)/2} t^{- \frac{1}{2} \sum_{i=1}^q \binom{\nu_i }{2}} (1+ o(1)),
\end{eqnarray*}
where $C_1$ depends only on drift vector $\bfa$.

\subsection{Proof of Theorem \ref{GBM.MainTh}.}
Following considerations of Section \ref{ss.asymptotic}, notice first that it suffices 
to take the first term from the sum \refs{firstterm} for asymptotic analysis
because next terms consists of positive rank polynomials of variable $\bfz$ and therefore they will tend to zero 
faster after substitution
\refs{podstawienie.xi}.
For the proof of the main theorem we have to plug the asymptotics \refs{finalI}
to integral \refs{GBM.caloscZSuma}.

\end{document}